\documentclass[11pt,leqno]{amsart}

\usepackage{amsmath}
\usepackage{amssymb}
\usepackage{a4wide}
\usepackage{soul}
\usepackage{graphicx}
\usepackage{graphics}
\usepackage{epsfig}
\usepackage{appendix}

\newcommand{\Subsection}[1]{\subsection{ #1} ${}^{}$}

\newtheorem{theorem}{Theorem}[section]
\newtheorem{lemma}[theorem]{Lemma}
\newtheorem{proposition}[theorem]{Proposition}
\newtheorem{definition}[theorem]{Definition}

\newcounter{hypo}

\newenvironment{hyp}{
 \begin{enumerate}
\setcounter{enumi}{\value{hypo}} \item}{\stepcounter{hypo} \end{enumerate}}

\makeatletter
 \@addtoreset{equation}{section}
 \makeatother
 
\title[SSF for perturbed periodic Schr\"odinger operators]
{Spectral shift function for perturbed periodic Schr\"odinger
operators.\\ \textrm{The large-coupling constant limit case}.}

\author[M. Dimassi]{Mouez Dimassi}
\address{Mouez Dimassi, LAGA, (UMR CNRS 7539), Univ. Paris 13, F-93430 Villetaneuse, France}
\email{dimassi@math.univ-paris13.fr}
\author[M. Zerzeri]{Maher Zerzeri}
\address{Maher Zerzeri, LAGA, (UMR CNRS 7539), Univ. Paris 13, F-93430 Villetaneuse, France}
\email{zerzeri@math.univ-paris13.fr}


\begin{document}

\begin{abstract}
In the large coupling constant limit, we obtain  an asymptotic
expansion in powers of $\mu^{-\frac{1}{\delta}}$ of the derivative
of the spectral shift function corresponding to the pair
$\big(P_\mu=P_0+\mu W(x),P_0=-\Delta+V(x)\big),$ where $W(x)$ is
positive, $W(x)\sim w_0(\frac{x}{|x|})|x|^{-\delta}$ near infinity
for some $\delta>n$ and $w_0\in {\mathcal C}^\infty(\mathbb
S^{n-1};\,\mathbb R_+).$ Here $\mathbb S^{n-1}$ is the unite sphere
of the space $\mathbb R^n$ and $\mu$ is a large parameter. The
potential $V$ is real-valued, smooth and periodic with respect to a
lattice $\Gamma$ in ${\mathbb R}^n$.
\end{abstract}

\maketitle

\setcounter{tocdepth}{2}
\tableofcontents

\vfill\break
\section{Introduction}
Consider the perturbed periodic Schr\"odinger operator
\begin{equation}
P_{\mu}=P_0+\mu W(x),\quad \mu>0,
\end{equation}
$$
P_0=-\Delta_x+V(x).
$$
Here $V$ is a real-valued, ${\mathcal C}^\infty$  function and
periodic with respect to a lattice $\Gamma$ of $\mathbb R^n$. We
assume  that $W\in {\mathcal C}^\infty({\mathbb R}^n; {\mathbb R})$
and satisfies the following estimate: for all $\alpha\in \mathbb{N}^n,$ there exists $C_\alpha>0$ such that
\begin{equation}\label{as1}
\vert \partial^\alpha_x W(x)\vert \leq C_\alpha(1+\vert x\vert)^{-\delta-|\alpha|},\quad \forall x\in \mathbb R^n,\quad\text{with} \  {\delta>n}.
\end{equation}

The operators $P_0, P_\mu$ are self-adjoint on $H^2({\mathbb R}^n)$.
Under the assumption (\ref{as1}) we show in Theorem \ref{w-SSFmu} below
that the operator $\big[f(P_\mu)-f(P_0)\big]$ belongs to the trace class for all
$f\in {\mathcal C}^\infty_0(\mathbb R)$. Following the general setup we
define the spectral shift function, SSF,
$\xi_\mu(\lambda):=\xi(\lambda;P_\mu,P_0)$ related to the pair $(P_\mu,P_0)$ by
\begin{equation}
{\rm tr}\big[f(P_\mu)-f(P_0)\big]=-\langle \xi_\mu'(\cdot),
f(\cdot)\rangle=\int_{\mathbb R} \xi_\mu(\lambda) f'(\lambda)
d\lambda,\quad \forall f\in {\mathcal C}^\infty_0(\mathbb R).
\end{equation}
By this formula $\xi_\mu$ is defined modulo a constant but for the
analysis of the derivative $\xi'_\mu(\lambda)$ this is not
important.

The notion of SSF was first singled out by the outstanding
theoretical physicist I-M.Lifshits in his investigations in the
solid state theory, in 1952, see \cite{Li52_01}. It was brought
into mathematical use in M-G. Kre\u{\i}n's famous paper
\cite{Kr53_01}, where the precise statement of the problem has
given and explicit representation of the SSF in term of the
perturbation determinant was obtained. The work of M-G.
Kre\u{\i}n's on the SSF has been described in detail in
\cite{BiYa93_01}. Background information on the SSF theory can
be found in \cite{Ro99_01} and  \cite[Chapter 8]{Ya00_01}.

In the case where $V=0$, the asymptotic behavior of the SSF of the
Schr\"odinger operator has been intensively studied in different
aspects (see \cite{Co81_01,Gu85_01,MaRa78_01,PePo82_01,Ro92_01,Ro94_01,RoTa87_01}
and the references given there).

In the semi-classical regime (i.e. $H(h)=-h^2\Delta_x+W(x),
(h\searrow 0))$ a Weyl type asymptotics of
$\xi_h(\lambda)=\xi(\lambda; H(h), -h^2\Delta)$ with sharp remainder
estimate has been obtained (see \cite{Ro92_01,Ro94_01,RoTa87_01,RoTa88_01}).
On the other hand, if an energy
$\lambda>0$ is non-trapping for the classical hamiltonian
$p(x,\zeta)=\vert \zeta\vert^2+W(x)$ (i.e. for all $(x,\zeta)\in
p^{-1}\{\lambda\}$, $\vert {\rm exp}(tH_p(x,\zeta)\vert \rightarrow
\infty$ when $t\rightarrow \infty$) a complete asymptotic
expansion in powers of $h$  of $\xi'_h(\lambda)$ has been
obtained (see \cite{Ro92_01,Ro94_01,RoTa87_01,RoTa88_01}).
Similar results are well-known for the SSF at
high energy (see \cite{Bu71_01,Co81_01,PePo82_01,Po82_01,Ro91_01}).

In the large coupling constant limit, the asymptotic behavior of
$\xi_\mu(\lambda):=\xi(\lambda; -\Delta+\mu
W, -\Delta)$ depends both on the sign of
the perturbation $W$ and on its decay properties at infinity.
For the case of non-negative perturbation $W\geq 0$ satisfying
$W(x)\sim w_0(\frac{x}{|x|})|x|^{-\delta}$ near infinity for some
$\delta>n$, it has been proved in \cite{PuRu02_01} (see also \cite{PuRu02_02}) that
$$
\xi_\mu(\lambda)=\mu^{\frac{n}{\delta}}\big(b_0+o(1)\big),\quad \mu\rightarrow +\infty,
$$
$$
b_0=(2\pi)^{-n}\kappa_0\int_{\mathbb R^n}\Big((\lambda)^{\frac{n}{2}}_+-
(\lambda-w_0(\frac{x}{|x|})|x|^{-\delta})^{\frac{n}{2}}_+\Big)\,dx,
$$
where $\kappa_0={\rm vol}(\{x\in \mathbb R^n;\ |x|<1\})$ and $(\lambda)_+=\max(\lambda,0)$.

Under the assumption that $\omega_0>0$ on $\mathbb S^{n-1}$ a complete
asymptotic expansion in powers of $\mu^{-\frac{1}{\delta}}$ is obtained in
\cite{Di06_01}.

In the literature there are a lot of works concerning periodic
Schr\"odinger operator with perturbations see \cite{AlDeHe89_01, BiYa95_01,
Bu87_01, Di98_01, Di02_01, Di05_01, DiZe03_01, Ge90_01,
GeMaSj91_01, GeNi98_01, GuRaTr88_01, HoSpTe01_01, Sl49_01} but there are
only few ones dealing  with the spectral shift function, see \cite{BiYa95_01}, \cite{Di05_01} and
also \cite{GeNi98_01}.

It should be mentioned that the tools in \cite{Di05_01}  are related to the
asymptotic behavior of $\xi(\lambda;P_0+W(hx),P_0),\ h\searrow 0$.

To our best knowledge there are no works treating the large
coupling constant limit in the case where $V\not=0$.
The goal of this work is to generalize
the results of \cite{Di06_01} to the perturbed periodic Schr\"odinger
operator $P_\mu=P_0+\mu W(x)$.

{\sl The paper is organized as follows}: In the next section, we
recall some well-known results concerning the spectra of a periodic
Schr\"odinger operator (Subsection \ref{premi}) and we state the
assumptions and the results precisely (Subsection \ref{main}).
We give an outline of the proofs in Subsection
\ref{outproof}. Section \ref{proofs} is devoted to the proofs.
In Subsection \ref{semiRO} we built a semiclassical reference operator, denoted by
$Q:=H(\mu^{-\frac{1}{\delta}}),$
that we use in all the rest of the paper.
The proof of the weak asymptotic expansion of $\xi'_\mu$ is given
in Subsection \ref{pw-SSFmu}.
At last, The pointwise asymptotic expansion of $\xi'_\mu$
is proved in Subsection \ref{p-dSSFmu}.

\section{Statements}

\Subsection{Preliminaries}\label{premi}

Let $\Gamma=\underset{i=1}{\overset{n}{\oplus}}{\mathbb Z}e_i$ be a lattice generated
by some basis $(e_1,e_2,\cdots,e_n)$ of ${\mathbb R}^n.$
The dual lattice $\Gamma^*$ is given by
$\Gamma^*:=\{\gamma^*\in {\mathbb R}^n;\
\langle \gamma | \gamma^* \rangle \in 2\pi{\mathbb Z},\
\forall \gamma\in \Gamma\}.$
A fundamental domain of $\Gamma$ (resp. $\Gamma^*$)
is denoted by $E$ (resp. $E^*$). If we identify opposite edges
of $E$ (resp. $E^*$) then it becomes a flat torus
denoted by $\displaystyle {\mathbb T}={\mathbb R}^n/\Gamma$ (resp.
$\displaystyle {\mathbb T}^*={\mathbb R}^n/\Gamma^*).$

Let $V$ be a real-valued potential, ${\mathcal C}^\infty$ and
$\Gamma-$periodic. For $k\in {\mathbb R}^n,$ we define the
operator $P(k)$ on $L^ 2({\mathbb T})$ by
$P(k):=(D_{y}+k)^2+V(y).$ The operator $P(k)$ is a
semi-bounded self-adjoint with $k$-independent domain
$H^ 2({\mathbb T}).$ Since the resolvent of $P(k)$ is compact,
$P(k)$ has a complete set of (normalized) eigenfunctions
$\Phi_{n}(\cdot,k)\in H^{2}({\mathbb T}),\ n\in {\mathbb N},$
called Bloch functions. The corresponding eigenvalues accumulate at
infinity and we enumerate them according to their multiplicities,
$\lambda_{1}(k)\leq \lambda_{2}(k)\leq \cdots .$ The operator
$P(k)$ satisfies the identity $e^{-iy\cdot
\gamma^*}P(k)e^{iy\cdot \gamma^*}=P(k+\gamma^*),\ \forall
\gamma^*\in \Gamma^*,$ then for every $p\geq 1,$ the function
$k\mapsto \lambda_{p}(k)$ is $\Gamma^{*}-$periodic.

Ordinary perturbation theory shows that $\lambda_{p}(k)$ are
continuous functions of $k$ for any fixed $p,$ and
$\lambda_{p}(k)$ is even an analytic function of $k$ near any
point $k_{0}\in  {\mathbb T}^{*}$ where $\lambda_{p}(k_{0})$ is
a simple eigenvalue of $P(k_{0}).$ The function $\lambda_{p}(k)$
is called the band function and the closed intervals
$\Lambda_{p}:=\lambda_{p}({\mathbb T}^*)$ are called bands. See
\cite{ReSi78_01}, \cite{Sj91_01} and also \cite{Sk85_01, Sk85_02}.

Consider the self-adjoint operator on $L^{2}({\mathbb R}^n)$
with domain $H^{2}({\mathbb R}^n)$:
\begin{equation}\label{freeoperator}
P_{0}=-\Delta_x +V(x),\quad \text{where}\ \Delta_x=\sum_{j=1}^n\frac{\partial^2}{\partial x_j^2}.
\end{equation}

The spectrum of $P_0$ is absolutely continuous
(see \cite{Th73_01}) and consists of the bands $\Lambda_{p},\ p=1,2,\cdots$.
Indeed,
$\displaystyle
\sigma(P_{0})=\sigma_{\rm ac}(P_{0})=\underset{p\geq
1}{\cup}\Lambda_{p}.$ See also \cite{Sh79_01}.

\begin{definition}\label{Fermisurface}
Let $\lambda\in {\mathbb R}$ and $F(\lambda)=\big\{k\in {\mathbb
T}^*;\ \lambda\in \sigma\big(P(k)\big)\big\}$ the corresponding
Fermi-surface.
\begin{itemize}
\item[a)] We will say that $\lambda\in \sigma(P_{0})$ is a simple energy level
if and only if $\lambda$ is a simple eigenvalue of $P(k),$ for
every  $k\in F(\lambda).$

\item[b)] Assume that $\lambda$ is a simple energy level of $P_{0}$
and let $\lambda(k)$ be the unique eigenvalue defined on a
neighborhood of $F(\lambda)$ such that $\lambda(k)=\lambda, \
\forall k\in F(\lambda).$ We say that $\lambda$ is a
non-critical energy  of $P_0$ if $d_{k}\lambda(k)\not=0$ for
all $k\in F(\lambda).$
\end{itemize}
\end{definition}

Note that in one dimension case $F(\lambda)$ is just a finite set of points.

Now, let us recall some well-known facts about the density of states
associated with $P_0,$ see \cite{Sh79_01}.
The density of states measure $\rho$
is defined as follows:
\begin{equation}
\rho(\lambda):=\frac{1}{(2\pi)^n}\sum_{p\geq 1}
\int_{\{k\in E^*;\ \lambda_p(k)\leq \lambda\}}\, dk.
\end{equation}

Since the spectrum of $P_{0}$ is absolutely continuous,
the measure $\rho$ is absolutely continuous with respect to
the Lebesgue measure $d\lambda.$ Therefore
the density of states, $\frac{d\rho}{dE}(E),$ of $P_{0}$
is locally integrable.

\Subsection{Results}\label{main}

Now, we introduce our perturbed periodic Schr\"odinger operator precisely:
\begin{equation}
P_\mu:=P_0+\mu W(x),\quad \mu>0,
\end{equation}
where $P_0$ is a periodic Schr\"odinger operator given in \eqref{freeoperator}
and $W\in {\mathcal C}^\infty(\mathbb R^n;\mathbb R)$.
Assume that:
\begin{hyp}\label{mu1}
 $W$ is strictly positive
\end{hyp}
and satisfying the following condition:
\begin{hyp}\label{mu2}
There exists a sequence $(w_j)_{j\geq 0}\subset {\mathcal C}^\infty(\mathbb S^{n-1};\mathbb R)$ such that
for all integer $N,$ there exists $R_N(x)\in {\mathcal C}^\infty(\mathbb R^n;\mathbb R)$ s.t.

\begin{center}
$\displaystyle W(x)=\sum_{j=0}^Nw_j\big(\frac{x}{|x|}\big)|x|^{-\delta-j}+R_N(x),$
for all $x,\ |x|\geq 1,$
\end{center}
where \fbox{${\delta>n}$} and for all $\alpha\in \mathbb N^n,$ there exists $C_\alpha>0$ such that

\begin{center}
$|\partial_x^\alpha R_N(x)|\leq C_\alpha (1+|x|)^{-\delta-N-1-|\alpha|}.$
\end{center}
\end{hyp}

\begin{hyp}\label{mu3}
Assume also that $w_0>0.$
\end{hyp}

The operator $P_\mu$ is self-adjoint, semi-bounded on $L^2(\mathbb
R^n)$ with domain $H^2(\mathbb R^n).$
\\
The assumption \ref{mu2} implies that $W$ goes to zero at infinity then
by perturbation theory (Weyl theorem) yields:
\begin{equation}
\sigma_{\rm ess}\big(P_\mu\big)=\sigma_{\rm
ess}(P_0)=\sigma(P_0)=\bigcup_{p\geq 1}\Lambda_p.
\end{equation}

Recall that $\sigma_{\rm ess}(A),$ the essential spectrum of $A,$ is
defined by $\sigma_{\rm ess}(A)=\sigma(A)\setminus \sigma_{\rm
disc}(A),$ where $\sigma_{\rm disc}(A)$  is the set of isolated
eigenvalues of $A$ with finite multiplicity. Here $A$ is an
unbounded operator on a Hilbert space.

Our first theorem concerns the weak asymptotic of
$\xi'_\mu(\lambda).$
\begin{theorem}[Weak asymptotic]\label{w-SSFmu}
Let $I$ be a bounded open interval on $\mathbb R.$ Assume that $W$
satisfies \ref{mu1}, \ref{mu2} and \ref{mu3}. For $f\in
{\mathcal C}_0^\infty(I),$ the operator $\Big[f(P_\mu)-f(P_{0})\Big]$ is of
trace class and
\begin{equation}\label{a1}
-\langle\xi'_\mu,f\rangle:=
{\rm tr}\Big[f(P_\mu)-f(P_{0})\Big]\sim
\mu^{\frac{n}{\delta}}\sum_{j=0}^{+\infty}a_{j}(f)\mu^{-\frac{j}{\delta}},
\quad \text{when}\ \mu\uparrow +\infty,
\end{equation}
with
\begin{equation}\label{a2}
a_{0}(f)=(2\pi)^{-n}\sum_{p\geq 1}\int_{{\mathbb R}^n_x}\int_ {E^*}
\Big[f\big(\lambda_p(k)+w_0(\frac{x}{|x|})|x|^{-\delta}\big)-
f\big(\lambda_p(k)\big)\Big]\,dk\,dx
\end{equation}
and
\begin{equation}
a_{1}(f)=(2\pi)^{-n}\sum_{p\geq 1}\int_{{\mathbb R}^n_x}\int_ {E^*}
f'\Big(\lambda_p(k)+w_0(\frac{x}{|x|})|x|^{-\delta}\Big)w_1\Big(\frac{x}
{|x|}\Big)|x|^{-\delta-1}\,dk\,dx.
\end{equation}

The coefficients $a_j(f)$ are distributions on $f.$

Moreover, if $\lambda$ is a non-critical energy of $P_0$ for all
$\lambda\in I,$ then $a_j(f)=-\langle \gamma_j(\cdot),f\rangle,$ for
all $f\in {\mathcal C}_0^\infty(I).$ Here $\gamma_j(\lambda)$ are smooth
functions of $\lambda\in I$. In particular,
\begin{equation}\label{gam0}
\gamma_0(\lambda)={\frac{\rm d}{{\rm d}\lambda}}\left[\int_{\mathbb
R^n_x}\Big\{\rho\big(\lambda\big)-\rho\Big(\lambda-w_0\big(\frac{x}
{|x|}\big)|x|^{-\delta}\Big)\Big\} \, dx\right]
\end{equation}
and
\begin{equation}\label{gam1}
\gamma_1(\lambda)={\frac{{\rm d}^2}{{\rm d}\lambda^2}}\left[
\int_{\mathbb R^n_x}\rho\Big(\lambda-w_0\big(\frac{x}
{|x|}\big)|x|^{-\delta}\Big)w_1\big(\frac{x}
{|x|}\big)|x|^{-\delta-1}\, dx\right].
\end{equation}
\end{theorem}

The proof of this theorem is contained in Subsection \ref{pw-SSFmu}.

Our main result concerning the derivative of the spectral
shift function near the bottom of the spectrum is the following.

Let $\lambda_0=\inf\sigma(P_0)\in \mathbb R.$ Let $\lambda_1(k)$ be
the first Floquet eigenvalue. It is well-known (see
\cite{ReSi78_01}, \cite{Sk85_01, Sk85_02}) that there exists a
bounded interval $[a,b]\subset\lambda_1({\mathbb T}^{*})\subset
\sigma(P_0)$ near $\lambda_0$ such that for all
$\lambda=\lambda_1(k)\in [a,b],$ $\lambda$ is a non-critical energy
of $P_0$ and satisfies:
$$
\Delta_{k}\lambda_1(k)>0,\quad \text{for all}\ \
k\in{F}(\lambda).\leqno{\textbf{(*)}}
$$
Recall that $F(\lambda)$ is a Fermi-surface associated to $\lambda.$

\begin{theorem}[Pointwise asymptotic]\label{dSSFmu}
Fix $[a,b]$ as above $($satisfying {\bf (*)}$)$. Assume \ref{mu1}, \ref{mu2} and \ref{mu3}.
Then the following asymptotic expansion holds:
\begin{equation}\label{strongSSFmu}
\xi'_\mu(\lambda)\sim \mu^{\frac{n}{\delta}}\sum_{j\geq
0}\gamma_j(\lambda)\mu^{-\frac{j}{\delta}}\quad {\rm as}\
\mu\uparrow +\infty,
\end{equation}
uniformly for $\lambda\in [a,b].$ \\
The coefficients $\big(\gamma_j(\lambda)\big)_{j\geq 0}$ are given in Theorem \ref{w-SSFmu}.
Furthermore, this expansion has derivate in $\lambda$ to any order.
\end{theorem}

\Subsection{Outline of the proofs}\label{outproof}

The purpose of this section is  to provide  a broad outline of the
proofs. As we have noticed in the introduction the asymptotics like
(\ref{a1}) and (\ref{strongSSFmu}) are well-known for $\xi'(\lambda;
P_0+W(hx), P_0)$ where $h=\mu^{-\frac{1}{\delta}}\downarrow 0$ and
$W$ is a regular potential, bounded with all its derivatives (see
\cite{DiZe03_01, DiZe10_01}). Our strategy in this work will be to
show that:
\begin{equation}\label{reduc1}
\xi'_\mu(\lambda)=\xi'(\lambda;Q,P_0)+ {\mathcal O}(\mu^{-\infty}),
\end{equation}
where $Q:=H(\mu^{-\frac{1}{\delta}})=
P_0+\varphi(\mu^{-\frac{1}{\delta}}x,\mu^{-\frac{1}{\delta}})$ and
\begin{equation}\label{asy}
\varphi(x,h)=\phi_0(x)+h\phi_1(x)+\cdots,
\end{equation}
has a full asymptotic expansion in powers of $h:=\mu^{-\frac{1}{\delta}},$
with coefficients are uniformly bounded on $x$ together with their derivatives.
This makes it possible to apply the results of \cite{DiZe03_01}
and \cite{DiZe10_01} (see also \cite{Di05_01}).

The main idea of the proof of (\ref{reduc1}) are the two following facts.

{\bf (1)} Since $W>0$ by assymptions \ref{mu1} and  \ref{mu2}, it follows that for all
$C>0$ there exists $\mu_C>>1$ such that for $\mu>\mu_C$ the set
$$
\big\{(x,\zeta)\in {\mathbb R}^{2n};\ \vert x\vert <C\ \hbox{\rm and} \
\zeta^2+V(x)+\mu W(x) \in [a,b]\big\}=\emptyset.
$$
Here $[a,b]$ is any bounded interval of $\mathbb R$. Thus, on the symbol level,
${\rm tr}\big(f(P_\mu)-f(P_0)\big)$ for $f\in {\mathcal C}^\infty_0({\mathbb R})$,
only depends on the asymptotic behavior of $W(x)$ at infinity.

{\bf (2)} On the other hand, by assumption \ref{mu2},
$\mu W(x)=\varphi(hx,h)$, where $\varphi(x,h)$ satisfies (\ref{asy}) with $h=\mu^{-\frac{1}{\delta}}$
and $\phi_j(x)=w_j(\frac{x}{\vert x\vert})\vert x\vert^{-\delta-j},\ j=1,2,\cdots $ for $|x|>>1$.

Armed with the above remarks, we will construct in the next subsection a semiclassical reference
operator $Q:=H(\mu^{-\frac{1}{\delta}})=P_0+\varphi(hx;h)$ with
$\varphi(x,h)$ satisfies (\ref{asy}) (see Lemma \ref{varphi}) and
satisfying the following crucial identity:
\begin{equation}\label{genidentity}
(z-P_\mu)^{-1}-(z-Q)^{-1}=A_1(z)+A_2(z)(z-Q)^{-1}+ (z-P_\mu)^{-1}A_3(z)(z-Q)^{-1},
\end{equation}
where $z\rightarrow A_i(z),\ (i=1,2,3)$ is analytic in a complex
neighborhood ${\mathcal U}$ of $[a,b]$. Moreover, $A_2(z)$ and
$A_3(z)$ are of trace class and for all $s,s'\in {\mathbb R}$, we
have
\begin{equation}\label{negli}
\Big\Vert \langle hx\rangle^s A_j(z) \langle hx\rangle^{s'}\Big\Vert_{\rm tr}
={\mathcal O}\big(\mu^{-\infty}\big),\quad j=2,3
\end{equation}
uniformly on $z\in {\mathcal U}$ (see Lemma
\ref{negligeabletraceoperator1} and Proposition
\ref{traceestimate}). We recall that $h=\mu^ {-\frac{1}{\delta}}$.

Using (\ref{genidentity}) and the Helffer-Sj\"ostrand formula we get (\ref{reduc1}) in the
sense of distributions for all $f\in {\mathcal C}^\infty_0(]a,b[)$. Thus,
Theorem \ref{w-SSFmu} follows from Theorem 1.3 in \cite{DiZe03_01}.

On the other hand, since $z\rightarrow A_i(z),\ (i=1,2,3)$ are analytic,
it follows from (\ref{genidentity}) and the stone's formula
(see also Proposition \ref{mainresultofthesubsection}) that
\begin{align}\label{mm}
& \xi'_\mu(\lambda)-\xi'(\lambda;Q,P_0)={\rm tr}\Big(A_2(\lambda) \big[(\lambda+i0-P_\mu)^{-1}
-(\lambda-i0-P_\mu)^{-1}\big]\Big)+\\
& {\rm tr}\Big(\big[(\lambda+i0-P_\mu)^{-1}-(\lambda-i0-P_\mu)^{-1}\big]A_3(\lambda)
\big[(\lambda+i0-Q)^{-1}-(\lambda-i0-Q)^{-1}\big]\Big)\nonumber\\
& =(A)+(B).\nonumber
\end{align}

Using the cyclicity of the trace, we obtain
\begin{equation}
(A)={\rm tr}\Big(\big[\langle hx\rangle^s A_2(\lambda) \langle hx\rangle^s\big]
\langle hx\rangle^{-s}\big[(\lambda+i0-P_\mu)^{-1}-(\lambda-i0-P_\mu)^{-1}\big]\langle hx\rangle^{-s}\Big).
\end{equation}
Using {\bf (*)}, we will prove that: for $s>\frac{1}{2}$
\begin{equation}\label{abs}
\big\Vert\langle hx\rangle^{-s}(\lambda\pm i0-Q)^{-1}\langle hx\rangle^{-s}\big\Vert={\mathcal O}(h^{-1}),\
\big\Vert\langle hx\rangle^{-s}(\lambda\pm i0-P_\mu)^{-1}\langle hx\rangle^{-s}\big\Vert={\mathcal O}(h^{-1}),
\end{equation}
uniformly on $\lambda\in [a,b]$. Recall that $h=\mu^{-\frac{1}{\delta}}$.
Combining (\ref{negli}) and (\ref{abs}) we get $(A)={\mathcal O}\big(\mu^{-\infty}\big)$ uniformly on
$\lambda\in [a,b]$. The same arguments give $(B)={\mathcal O}\big(\mu^{-\infty}\big)$.
Thus, it follows from (\ref{mm}) that (\ref{reduc1}) holds uniformly on $\lambda\in [a,b]$
and therefore Theorem \ref{dSSFmu} follows from \cite{DiZe10_01} (see also \cite[Theorem 3]{Di05_01}).

For completness, let us explain the proof of (\ref{abs}). Following \cite{GeMaSj91_01,GuRaTr88_01,HoSpTe01_01,Di93_01}
and the references given there, the spectral study of $(z-Q)$ for $z$ in a small complex neighborhood of
$\inf(\sigma(P_0))$ is reduced to the spectral study of an $h$-pseudodifferential operator
(called effective hamiltonian)
$$
E_{-+}(z,h)=z-\lambda_1(hD_x)-\phi_0(x)+{\mathcal O}(h),
$$
via the following well-known identity
$(z-Q)^{-1}=E(z)+E_+(z)E_{-+}(z,h)^{-1}E_-(z)$ where $E(z), E_+(z)$ and $E_-(z)$ are holomorphic
for $z$ in a small complex neighborhood of
$\inf(\sigma(P_0))$ (see \cite{SjZw07_01}, identity (3.15)).
Here $\phi_0$ is given in Lemma \ref{varphi}.

By construction of $\phi_0$ we have $\lambda_1(k)+\phi_0(x)\in [a,b]$ implies that
$\phi_0(x)=\omega_0(\frac{x}{\vert x\vert})\vert x\vert^{-\delta}$.
Combining this with {\bf (*)} we deduce that the interval $[a,b]$ is a non-trapping region
for the classical hamiltonian $p(k,x)=\lambda_1(k)+\phi_0(x)$.
Consequently, (\ref{abs}) follows from a standard limiting absorption principle (see \cite[Lemma 3.5]{RoTa88_01}
and also \cite{Ge08_01}).

\section{Proofs}\label{proofs}

\Subsection{Construction of semiclassical reference
operator $Q:=H(\mu^{-\frac{1}{\delta}})$}\label{semiRO}

In this subsection, we reduce the study of SSF
$\xi'_\mu(\cdot)$ to a semiclassical problem via the
following formula: For $f\in {\mathcal C}_0^\infty(\mathbb R),$
$$
{\rm tr}\big[f(P_\mu)-f(P_0)\big]={\rm tr}\big[f(Q)-f(P_0)\big]+{\mathcal O}(\mu^{-\infty}),\ \mu>>1,
$$
where $Q:=H(\mu^{-\frac{1}{\delta}})=P_0+\varphi(\mu^{-\frac{1}{\delta}}x,\mu^{-\frac{1}{\delta}})$ and
$\varphi(x,\mu^{-\frac{1}{\delta}})$ has a full asymptotic expansion on $\mu^{-\frac{1}{\delta}},$
with coefficients are uniformly bounded on $x$ together with their derivative.
Moreover $\varphi(\mu^{-\frac{1}{\delta}}x ,\mu^{-\frac{1}{\delta}})$ coincide with
$\mu W(x)$ outside $\Omega_M(\mu^{-\frac{1}{\delta}})=\{x\in \mathbb R^n;\ \mu W(x)\geq M\}$ for $M$
sufficiently large.

Set $h=\mu^{-\frac{1}{\delta}}.$ For $M>0,$ put $\Omega_M(h)=\{x\in \mathbb R^n;\ h^{-\delta}W(x)>M\}.$
Let $r_1, r_2$ two real satisfying $\displaystyle 0<r_1<\big(\min_{\mathbb S^{n-1}}w_0\big)^{\frac{1}{\delta}}
\leq \big(\max_{\mathbb S^{n-1}}w_0\big)^{\frac{1}{\delta}}<r_2.$ The two real $r_1, r_2$ exist since
$w_0>0$ and continuous on the unit sphere.

According to \ref{mu2} and \ref{mu3}, there exists $h_0>0$ such that
\begin{equation}
B\big(0,r_1M^{-\frac{1}{\delta}}h^{-1}\big)\subset \Omega_M(h)
\subset B\big(0,r_2M^{-\frac{1}{\delta}}h^{-1}\big),\quad \text{for all}\ \,
0<h\leq h_0.
\end{equation}
Here $B(0,r)$ denotes the ball of center 0 and radius $r$.

Let  $\chi\in {\mathcal C}_0^\infty\Big(B\big(0,r_1M^{-\frac{1}
{\delta}}\big);[0,1]\Big);\ \equiv 1$ near zero. We introduce the following quantities:
\begin{itemize}
 \item[(i)] $\displaystyle \varphi(x,h)=\big[1-\chi(x)\big]
h^{-\delta}W(\frac{x}{h})+M\chi(x),$

\item[(ii)] $\displaystyle \widetilde W(x)=h^{-\delta}W(x)-
\varphi(hx,h)=\chi(hx)\Big[h^{-\delta}W(x)-M\Big].$
\end{itemize}
By construction of $\varphi$ and $\widetilde W,$ we have:
\begin{lemma}\label{varphi}
The two functions $\varphi$ and $\widetilde W$ are ${\mathcal C}^\infty,$ and have the following properties:
$$
{\rm supp}(\widetilde W)\subset\ {\rm
supp}\big(\chi(h\cdot)\big)\subset\Omega_M(h),
$$
\begin{equation}\label{var-out}
\varphi(hx,h)>\frac{M}{2}\quad \text{for all}\ x\in
\Omega_{\frac{M}{2}}(h),
\end{equation}
\begin{equation}\label{var-est}
|\partial_x^\alpha\varphi(x,h)|\leq C_\alpha\quad \text{uniformly for}\ h\in
]0,h_0],
\end{equation}
and satisfies for all integer $N$, there exist
$\phi_0,\phi_1,\cdots,\phi_N, K_N(\cdot,h)\in C_b^\infty(\mathbb
R^n)$ uniformly bounded with respect to $h\in ]0,h_0]$ together with their derivatives
such that
$$
\varphi(x;h)=\sum_{j=0}^N\phi_j(x)h^j+h^{N+1}K_N(x,h).
$$
Moreover, the function $\phi_0$ is given by:
$$
\phi_0(x)=\big(1-\chi(x)\big)w_0\big(\frac{x}{|x|}\big)|x|^{-\delta}+M\chi(x).
$$
\end{lemma}

\begin{lemma}\label{for-phi0}
If $\phi_0(x)<M$ then $\phi_0(x)=w_0(\frac{x}{|x|})|x|^{-\delta}.$
\end{lemma}

In fact, if $x\in {\rm supp}\chi$ then $w_0(\frac{x}{|x|})|x|^{-\delta}>r_1^\delta |x|^{-\delta}>M,$
which implies that
$$
\phi_0(x)=\big(1-\chi(x)\big)w_0(\frac{x}{|x|})|x|^{-\delta}+M\chi(x)>M\quad \text{for all}\ x\in {\rm
supp}\chi.
$$
{\bf Choice of the constant $M$}: Let $I=[a,b]$ be a bounded interval of $\mathbb R.$
The first condition on $M,$ we choose $M>|a|+|b|$ such that $\lambda_p(k)+\varphi(x;h)\in I$ therefore
$\phi_0(x)<M$ for all $0<h\leq h_0.$

Let $\Theta\in {\mathcal C}^\infty\big(\mathbb R;[\frac{M}{3},+\infty[\big)$ satisfying
$\Theta(t)=t,$
for all $\displaystyle t\geq \frac{M}{2}.$ We define
$$
F_1(x;h)=\Theta\Big(\varphi(hx;h)\Big)\quad{\rm and}\quad
F_2(x;h)=\Theta\Big(h^{-\delta}W(x)\Big), \quad \text{for all}\ x\in \mathbb
R^n.
$$

Let ${\mathcal U}$ be a small complex neighborhood of $I.$
From now on, we choose $M$ large enough so that
$$
\sup_{x\in \mathbb R^n}\big[V(x)+F_i(x;h)-\Re(z)\big]\geq \frac{M}
{4}\quad \text{uniformly for}\ z\in{\mathcal U},\quad
i=1,2.
$$
This choice of $M$ implies that the function defined by
$z\mapsto\big(z-P_{F_i}\big)^{-1}$ is holomorphic
from ${\mathcal U}$ to ${\mathcal
L}\big(L^2(\mathbb R^n)\big),$ where $P_{F_i}=P_0+F_i,\ i=1,2.$
Moreover, it follows from \eqref{var-est} that
$\partial_x^\alpha F_i(x,h)={\mathcal O}_\alpha\big(h^{-\delta}\big).$

Finally \eqref{var-out} shows that
\begin{align}\label{disj-est}
{\rm dist}\Big({\rm supp}(\widetilde W),{\rm supp}\big[\varphi(hx,h)-F_1(x,h)\big]\Big)\geq \frac{a_1(M)}{h},&
\\
{\rm dist}\Big({\rm supp}(\widetilde W),{\rm supp}\big[h^{-\delta}W(x)-F_2(x,h)\big]\Big)
\geq \frac{a_2(M)}{h},&\nonumber
\end{align}
with $a_1(M), a_2(M)>0$ independent of $h.$

Let $Q:=H(\mu^{-\frac{1}{\delta}})=P_0+\varphi(hx;h).$ The operator $Q$ with domain $H^2(\mathbb
R^n)$ is self-adjoint.
\begin{proposition}\label{resolventidentity}
For all $z\in{\mathcal
U}\setminus\big[\sigma(P_\mu)\cup\sigma(Q)\big],$ we have
\begin{align}\label{reso-iden}
(z-P_\mu)^{-1}-(z-Q)^{-1} = &(z-P_{F_2})^{-1}\widetilde W(z-P_{F_1})^{-1}\\
&+ (z-P_{F_2})^{-1}\widetilde W(z-P_{F_1})^{-1}{\mathcal G}_1(z-Q)^{-1}\nonumber\\
&+ (z-P_\mu)^{-1}{\mathcal G}_2(z-P_{F_2})^{-1}\widetilde W(z-Q)^{-1},\nonumber
\end{align}
where ${\mathcal G}_1(\cdot;h)=\Phi\Big(\varphi\big(h(\cdot);h\big)\Big)$ and
${\mathcal G}_2(\cdot;h)=\Phi\Big(h^{-\delta} W(\cdot)\Big),$
with $\Phi(t)=t-\Theta(t),\ t\in \mathbb R.$
\end{proposition}

{\sl Proof.} From the resolvent equation, we have:
\begin{equation}\label{r1}
(z-P_\mu)^{-1}-(z-Q)^{-1}=(z-P_\mu)^{-1}\widetilde W(z-Q)^{-1},
\end{equation}
\begin{align}\label{r2}
(z-P_\mu)^{-1}
& =(z-P_{F_2})^{-1}+(z-P_\mu)^{-1}\Big[h^{-\delta}W(x)-F_2(x;h)\Big](z-P_{F_2})^{-1}\\
& =(z-P_{F_2})^{-1}+(z-P_\mu)^{-1}{\mathcal G}_2(z-P_{F_2})^{-1},\nonumber
\end{align}
and
\begin{align}\label{r3}
(z-Q)^{-1}
& =(z-P_{F_1})^{-1}+(z-Q)^{-1}\Big[\varphi(hx,h)-F_1(x;h)\Big](z-P_{F_1})^{-1}\\
& =(z-P_{F_1})^{-1}+(z-P_{F_1})^{-1}{\mathcal G}_1(z-Q)^{-1}.\nonumber
\end{align}

By inserting \eqref{r2} in the right-hand side of \eqref{r1}, we obtain
\begin{align}\label{r4}
(z-P_\mu)^{-1}-(z-Q)^{-1}
= &(z-P_{F_2})^{-1}\widetilde W(z-Q)^{-1}\\
& +(z-P_\mu)^{-1}{\mathcal G}_2(z-P_{F_2})^{-1}\widetilde W(z-Q)^{-1}.\nonumber
\end{align}
Now, we replace $(z-Q)^{-1}$
in the first term of the right-hand side of \eqref{r4} by \eqref{r3}, we get \eqref{reso-iden}.
\hfill$\square$

Note that the supports of ${\mathcal G}_1(\cdot;h)$ and ${\mathcal G}_2(\cdot;h)$ are contained in the complementary
of the set $\displaystyle \Omega_{\frac{M}{2}}(h).$

\begin{lemma}\label{negligeabletraceoperator1}
Let $\chi_1(\cdot,h), \chi_2(\cdot,h)$ be two $h-$dependent functions and ${\mathcal C}_b^\infty$
w.r.t $x$ in $\mathbb R^n.$ We assume that:\\
(i) the function $\chi_1$ is compactly supported, \\
(ii) there exists $C>0$ independent on $h$ such that ${\rm dist}\big({\rm supp}\chi_1,{\rm supp}\chi_2\big)\geq C>0,$\\
(iii) there exist $m_j$ such that for all $\alpha\in \mathbb N^n,$ $\partial_x^\alpha\chi_j(x,h)={\mathcal O}_\alpha\big(h^{-m_j}\big),$
$(j=1,2)$.\\
Then, for all two reals $s, s',$ the operator
$K_i:=\langle hx\rangle^s\big[\chi_1(hx,h)(z-P_{F_i})^{-1}\chi_2(hx,h)\big]
\langle hx\rangle^{s'}$ is of
trace class and it's trace norm is ${\mathcal O}(h^\infty)$
uniformly on $z\in {\mathcal U}.$ Here $i=1,2.$
\end{lemma}
Recall that $a(x)={\mathcal O}\big(\langle x\rangle^{-\infty}\big)$
means that for all integer $N$ there exist $C_N>0$ such that
$\big|a(x)\big|\leq C_N\big(\langle x\rangle^{-N}\big)$
and ${\mathcal C}_b^\infty$ is the space of ${\mathcal C}^\infty$ functions on $\mathbb R^n$ that are
uniformly bounded with respect to $h\in ]0,h_0]$ together with all their derivatives.

{\sl Proof.} Following Proposition 3.3 in \cite{Di94_01} and remembering $(i), (ii),$ the integral kernel
$k_i(x,y,z;h)$ of $K_i:=\langle hx\rangle^s\big[\chi_1(hx,h)(z-P_{F_i})^{-1}\chi_2(hx,h)\big]\langle hx\rangle^{s'},$
$(i=1,2)$, satisfies: for all $\alpha, \beta\in \mathbb N^n,$
$$
\partial_x^{\alpha}\partial_y^{\beta}\big[k_i(x,y,z;h)\big]={\mathcal O}_{\alpha,\beta}\left(h^\infty
\exp\Big[-\frac{1}{C}\big(d(x,{\rm supp}\chi_1(h\cdot,h))+d(y,{\rm supp}\chi_2(h\cdot,h))\big)\Big]\right),
$$
where $C$ and ${\mathcal O}_{\alpha,\beta}$ are independent of $(z,h)\in {\mathcal U}\times ]0,h_0]$ for some $h_0>0$ small.

Now, Lemma \ref{negligeabletraceoperator1} follows from the classical results on trace class operators
(see for instance \cite[Chapter 9]{DiSj99_01})
and the trace norm of $K_i$ can be estimated by
$$
\Vert K_i\Vert_{\rm tr}\leq c_n\sum_{|\alpha|+|\beta|\leq 2n+1}
\Vert \partial_x^\alpha\partial_y^\beta k_i(x,y,z;h)\Vert_{L^1(\mathbb R^{2n})},
$$
with a constant $c_n$ depending only on $n.$\hfill$\square$

For $z\in {\mathcal U},$ $\Im(z)\not=0,$ put
\begin{equation}\label{reste}
{\mathcal Q}(z)=(z-P_\mu)^{-1}-(z-Q)^{-1}-(z-P_{F_2})^{-1}\widetilde
W(z-P_{F_1})^{-1}.
\end{equation}

\begin{proposition}\label{traceestimate}
The operator ${\mathcal Q}(z)$ is of trace class and satisfies the
following estimate:
\begin{equation}\label{f41}
\Vert {\mathcal Q}(z)\Vert_{\rm tr}={\mathcal O}\big(h^\infty|\Im(z)|^{-2}\big),
\end{equation}
uniformly on $z\in {\mathcal U},\ \Im(z)\not=0.$
\end{proposition}

{\sl Proof.} Identity \eqref{reso-iden} gives
$$
{\mathcal Q}(z)=(z-P_{F_2})^{-1}\widetilde W(z-P_{F_1})^{-1}{\mathcal G}_1(z-Q)^{-1}
+(z-P_\mu)^{-1}{\mathcal G}_2(z-P_{F_2})^{-1}\widetilde W(z-Q)^{-1}.
$$
Applying the previous lemma and the fact that
$\Vert (z-Q)^{-1}\Vert, \Vert (z-P_\mu)^{-1}\Vert= {\mathcal O}(|\Im(z)|^{-1}),$
we get the proposition \ref{traceestimate}.
\hfill$\square$

Now we are ready to state the main result of this subsection. Let $J$
an open interval of $\mathbb R$ such that $I\subset\subset
J\subset\subset{\mathcal U}\cap\mathbb R.$

\begin{proposition}\label{mainresultofthesubsection}
The operator
$\big[f(P_\mu)-f(Q)\big]$ is of trace class with
${\mathcal O}\big(h^\infty\big)$ as trace norm, uniformly on $f\in
{\mathcal C}_0^\infty\big(J;\mathbb R\big);\ \equiv 1$ near of $\overline{I}.$ Moreover,
\begin{equation}\label{f42}
{\rm tr}\big[f(P_\mu)-f(Q)\big]=\lim_{\epsilon\downarrow 0}
\frac{i}{2\pi}\int_{\mathbb R} f(\lambda)
\Big[{\rm tr}\big({\mathcal Q}(\lambda+i\epsilon)\big)-
{\rm tr}\big({\mathcal Q}(\lambda-i\epsilon)\big)\Big]
\,d\lambda,
\end{equation}
where the limit is taken in the sense of distributions.
\end{proposition}

{\sl Proof.} Let $f\in {\mathcal C}_0^\infty\big(J;\mathbb R\big);\
\equiv 1$ near of $\overline{I}.$ Let $\tilde f\in {\mathcal
C}_0^\infty({\mathcal U})$ be an almost analytic extension of $f$,
i.e. $\tilde f_{|\mathbb R}=f$ and $\overline{\partial}\tilde
f(z)={\mathcal O}_N(|\Im(z)|^N)$ for all $N\in \mathbb N,$ here
$\overline{\partial}=\frac{\partial}{\partial \overline{z}}$. The
functional calculus due to Helffer-Sj\"ostrand (see for instance
\cite[Chapter 8]{DiSj99_01}) yields
$$
[f(P_\mu)-f(Q)]=-\frac{1}{\pi}\int\overline{\partial}\tilde f(z)
\Big((z-P_\mu)^{-1}-(z-Q)^{-1}\Big)\, L(dz),
$$
where $L(dz)=dxdy,$ $z=x+iy, (x,y)\in {\mathbb R}^2.$ Remembering
the definition of ${\mathcal Q}(z)$ with \eqref{reso-iden}, we
obtain
$$
[f(P_\mu)-f(Q)]=-\frac{1}{\pi}\int\overline{\partial}\tilde f(z)
(z-P_{F_2})^{-1}\widetilde W(z-P_{F_1})^{-1}\, L(dz)
-\frac{1}{\pi}\int\overline{\partial}\tilde f(z){\mathcal Q}(z)\, L(dz).
$$
Since $(z-P_{F_j})^{-1}, j=1,2,$ is holomorphic in a neighborhood of $\text{supp}(\tilde f),$
the first term in the r.h.s. of the previous identity vanishes. Consequently,
$$
[f(P_\mu)-f(Q)]=-\frac{1}{\pi}\int\overline{\partial}\tilde f(z){\mathcal Q}(z)\, L(dz).
$$
Combining this with Proposition \ref{traceestimate} and using \eqref{f41} and that
$\overline{\partial}\tilde f(z)={\mathcal O}\big(|\Im(z)|^2\big),$ we get
$\Vert f(P_\mu)-f(Q)\Vert_{\rm tr}={\mathcal O}\big(h^\infty\big)$ and
\begin{align}
& {\rm tr}\Big[f(P_\mu)-f(Q)\Big]=-\frac{1}{\pi}\int\overline{\partial}\tilde f(z){\rm tr}\big[{\mathcal Q}(z)\big]\, L(dz)
\nonumber\\
& =\lim_{\epsilon\searrow 0}\Big[-\frac{1}{\pi}\int_{\{\Im(z)>0\}}
\overline{\partial}\tilde f(z){\rm tr}\big[{\mathcal Q}(z+i\epsilon)\big]\, L(dz)-
\frac{1}{\pi}\int_{\{\Im(z)<0\}}
\overline{\partial}\tilde f(z){\rm tr}
\big[{\mathcal Q}(z-i\epsilon)\big]\, L(dz)
\Big].\nonumber
\end{align}

Using that ${\rm tr}\big[{\mathcal Q}(z+i\epsilon)\big]$ (resp. ${\rm tr}\big[{\mathcal Q}(z-i\epsilon)\big]$)
is holomorphic on $\{z\in {\mathcal U}, \Im(z)>0\}$ (resp. $\{z\in {\mathcal U}, \Im(z)<0\}$) and applying
Green formula, we have \eqref{f42}.
\hfill$\square$

\Subsection{Proof of the weak asymptotic expansion of $\xi'_\mu(\cdot)$}\label{pw-SSFmu}

Let $I$ be an interval of $\mathbb R$. Let $f\in {\mathcal C}_0^\infty(I),$ we have
$$
f(P_\mu)-f(P_0)=\Big[f(Q)-f(P_0)\Big]+\Big[f(P_\mu)-f(Q)\Big].
$$

Proposition \ref{mainresultofthesubsection} imply that
$\Big[f(P_\mu)-f(Q)\Big]$ is of trace class and $\big\Vert
[f(P_\mu)-f(Q)]\big\Vert_{\rm tr}={\mathcal O}(h^\infty).$ On the
other hand, applying Theorem 1.3 of \cite{DiZe03_01} to the operator
$Q$, with $h=\mu^{-\frac{1}{\delta}}\downarrow 0$ we obtain that
$\Big[f(Q)-f(P_0)\Big]$ is of trace class and
$$
{\rm tr}\Big[f(Q)-f(P_{0})\Big]\sim
h^{-n}\sum_{j=0}^{+\infty}a_{j}(f)h^{j},\quad \text{when}\
h\downarrow 0,
$$
with
\begin{equation}\nonumber
a_{0}(f)=(2\pi)^{-n}\sum_{p\geq 1}\int_{{\mathbb R}^n_x}\int_ {E^*}
\Big[f\big(\lambda_p(k)+\phi_0(x)\big)-
f\big(\lambda_p(k)\big)\Big]\,dk\,dx
\end{equation}
and
\begin{equation}\nonumber
a_{1}(f)=(2\pi)^{-n}\sum_{p\geq 1}\int_{{\mathbb R}^n_x}\int_ {E^*}
f'\Big(\lambda_p(k)+\phi_0(x)\Big)\phi_1(x)\,dk\,dx.
\end{equation}

Note that, the sum in these equalities is finite, since $\displaystyle \lim_{p\rightarrow +\infty}\lambda_p(k)=+\infty$
and $\phi_0$ is bounded.

The coefficient $a_j(f)$ is a finite sum of term of the form $\int\int c_l(x,k)f^{(l)}\big(b(x,k)\big)dxdk,$
where $c_k$ depends on $\phi_.$ and their derivatives and $b(x,k)\in\{\lambda_p(k),\lambda_p(k)+\phi_0(x)\},$
see \cite[Chapter 8, Identity (8.16)]{DiSj99_01}. Then
$$
{\rm tr}\Big(f(P_\mu)-f(P_0)\Big)={\rm tr}\Big(f(Q)-f(P_0)\Big)+{\mathcal O}\big(h^\infty\big)
$$
and recall that $h=\mu^{-\frac{1}{\delta}}\downarrow 0$, we get
$$
{\rm tr}\Big[f(P_\mu)-f(P_{0})\Big]\sim
\mu^{\frac{n}{\delta}}\sum_{j=0}^{+\infty}a_{j}(f)\mu^{-\frac{j}{\delta}},\quad \text{when}\
\mu\uparrow +\infty,
$$
with
\begin{equation}
a_{0}(f)=(2\pi)^{-n}\sum_{p\geq 1}\int_{{\mathbb R}^n_x}\int_ {E^*}
\Big[f\big(\lambda_p(k)+w_0(\frac{x}{|x|})|x|^{-\delta}\big)-
f\big(\lambda_p(k)\big)\Big]\,dk\,dx\nonumber
\end{equation}
and
\begin{equation}
a_{1}(f)=(2\pi)^{-n}\sum_{p\geq 1}\int_{{\mathbb R}^n_x}\int_ {E^*}
f'\Big(\lambda_p(k)+w_0(\frac{x}{|x|})|x|^{-\delta}\Big)w_1\Big(\frac{x}
{|x|}\Big)|x|^{-\delta-1}\,dk\,dx,\nonumber
\end{equation}
here we have used Lemma \ref{for-phi0} for the expression of $\phi_0(x)$.

If $\lambda$ is a non-critical energy of $P_0$ for all
$\lambda\in I.$  Then $d\big(\lambda_p(k)\big)\not=0$ and
$d\big(\lambda_p(k)+\phi_0(x)\big)\not=0$
for all $k\in F(\lambda).$ We recall that $F(\lambda)$ is the Fermi surface.
Therefore, $a_j(f)=-\langle \gamma_j(\cdot),f\rangle,$ for
all $f\in {\mathcal C}_0^\infty(I)$ and $\gamma_j(\lambda)$ are smooth
functions of $\lambda\in I,$ in particular,
$$
\gamma_0(\lambda)={\frac{\rm d}{{\rm d}\lambda}}\left[\int_{\mathbb
R^n_x}\Big\{\rho\big(\lambda\big)-\rho\Big(\lambda-w_0\big(\frac{x}
{|x|}\big)|x|^{-\delta}\Big)\Big\} \, dx\right]
$$
and
$$
\gamma_1(\lambda)={\frac{{\rm d}^2}{{\rm d}\lambda^2}}\left[
\int_{\mathbb R^n_x}\rho\Big(\lambda-w_0\big(\frac{x}
{|x|}\big)|x|^{-\delta}\Big)w_1\big(\frac{x}
{|x|}\big)|x|^{-\delta-1}\, dx\right],
$$
which complete the proof of the theorem \ref{w-SSFmu}.
\hfill$\square$

\Subsection{Proof of the pointwise asymptotic expansion of  $\xi'_\mu(\cdot)$}\label{p-dSSFmu}

Let $\lambda_0=\inf\sigma(P_0)\in \mathbb R.$ Let $\lambda_1(k)$ be
the first Floquet eigenvalue. As noticed in Subsection \ref{main}
there exists a bounded interval $[a,b]\subset\lambda_1({\mathbb
T}^{*})\subset \sigma(P_0)$ near $\lambda_0$ such that for all
$\lambda=\lambda_1(k)\in [a,b],$ $\lambda$ is a non-critical energy
of $P_0$ and satisfies:
$$
\Delta_{k}\lambda_1(k)>0,\quad \text{for all}\ \ k\in{F}(\lambda).
$$

From now on, we assume that $[a,b]$ satisfies the above properties.
Let $\mathcal U$ be a small complex neighborhood of $[a,b]$.
Put $\mathcal U_{\pm}=\mathcal U \cap \{z\in \mathbb C; \pm\Im(z)>0\}$.
Recall that $h=\mu^{-\frac{1}{\delta}}.$

\begin{proposition}[limiting absorption principle]\label{muabslimit2}
Assume \ref{mu1}, \ref{mu2} and
\ref{mu3}. For all $\alpha>l-\frac{1}{2}$ with $l\in {\mathbb N}^*,$ we have
\begin{equation}\label{abslimit2}
\big\Vert \langle hx\rangle^{-\alpha}(z-P_\mu)^{-l}\langle hx\rangle^{-\alpha} \big\Vert
={\mathcal O}(h^{-l}),
\end{equation}
uniformly on $z\in \mathcal U_{\pm}$.
\end{proposition}

\begin{proof}  First assume that the operator $Q$ satisfies the limiting absorption principle, i.e.
for all $\alpha>l-\frac{1}{2}$ with $l\in {\mathbb N}^*,$
\begin{equation}\label{habslimit}
 \big\Vert \langle hx\rangle^{-\alpha}(z-Q)^{-l}\langle hx\rangle^{-\alpha} \big\Vert
={\mathcal O}(h^{-l}),
\end{equation}
uniformly on $z\in \mathcal U_{\pm}$.
Then we prove the estimation (\ref{abslimit2}). In fact,
it follows from identity \eqref{r4} that for all
$z\in{\mathcal U}\setminus\big[\sigma(Q)\cup\sigma(P_\mu)\big],$
$$
(z-P_\mu)^{-1}\big[{\rm Id}-{\mathcal G}_2(z-P_{F_2})^{-1}\widetilde W(z-Q)^{-1}\big]
=\big[{\rm Id}+(z-P_{F_2})^{-1}\widetilde W\big](z-Q)^{-1}.
$$
Then, the weighted resolvent $(z-P_\mu)^{-1}_\alpha:=\langle hx\rangle^{-\alpha}(z-P_\mu)^{-1}\langle hx\rangle^{-\alpha}$
satisfies the identity:
\begin{align}\label{weightre}
(z-P_\mu)^{-1}_{\alpha}\Big[{\rm Id}-\big({\mathcal G}_2(z-P_{F_2})^{-1}\widetilde W\big)_{-\alpha}&(z-Q)^{-1}_{\alpha}\Big]
=\\
& \big[{\rm Id}+\langle hx\rangle^{-\alpha}(z-P_{F_2})^{-1}\widetilde W\langle hx\rangle^{\alpha}\big](z-Q)^{-1}_{\alpha}.\nonumber
\end{align}
Here, we used the notation $A_\gamma:=<hx>^{-\gamma}A<hx>^{-\gamma}.$

Using the fact that $z\mapsto (z-P_{F_2})^{-1}$ is holomorphic, $\widetilde W $ is compactly supported and the estimation
\eqref{habslimit} for $(z-Q)^{-1}_\alpha$ we conclude that for all $\alpha>\frac{1}{2}$ the right hand side of \eqref{weightre} is
${\mathcal O}(h^{-1})$ uniformly on $z\in {\mathcal U}_{\pm}$ as a bounded operator from $L^2(\mathbb R^n)$ on itself.
Note that by classical result (see \cite{ReSi78_01}) the integral kernel $r_2(x,y,z;h)$ of $(z-P_{F_2})^{-1}$ satisfies:
for all $\beta,\tilde\beta\in \mathbb N^n,$
$$
\partial_x^{\beta}\partial_y^{\tilde\beta}r_2(x,y,z;h)={\mathcal O}_{\beta,\tilde \beta}(h^{-\delta})e^{-c|x-y|},\quad |x-y|>1,
$$
where $c$ and ${\mathcal O}_{\beta,\tilde \beta}$ are independent of $(z,h)\in {\mathcal U}\times ]0,h_0].$

On the other hand, Lemma \ref{negligeabletraceoperator1} yields that $\big({\mathcal G}_2(z-P_{F_2})^{-1}\widetilde W\big)_{-\alpha}$ is
${\mathcal O}(h^\infty)$ as a bounded operator from $L^2(\mathbb R^n)$ on itself uniformly on $z\in {\mathcal U}$. Then with \eqref{habslimit}
the operator $\Big[{\rm Id}-\big({\mathcal G}_2(z-P_{F_2})^{-1}\widetilde W\big)_{-\alpha}(z-Q)^{-1}_{\alpha}\Big]$ is invertible as a bounded
operator from $L^2(\mathbb R^n)$ on itself uniformly on $z\in {\mathcal U}_{\pm}$ for $h\in ]0,h_0], h_0>0$ small and his inverse is ${\mathcal O}(1)$
as a bounded operator from $L^2(\mathbb R^n)$ on itself uniformly on $z\in {\mathcal U}_{\pm}$ for $h\in ]0,h_0], h_0>0$.

Hence, for all $\alpha>\frac{1}{2}$
$(z-P_\mu)^{-1}_{\alpha}={\mathcal O}(h^{-1})$ as a bounded operator from $L^2(\mathbb R^n)$
on itself uniformly on $z\in {\mathcal U}_{\pm}$ for $h\in ]0,h_0], h_0>0$. The same argument work for $l\geq 2$.

It is well-known that by the method of effective hamiltonian spectral problems of $Q=H(h)$
can be reduced to similar problem of systems of $h$-pseudodifferential operators, (see \cite{GeMaSj91_01,Di93_01,DiZe03_01}).
In our situation the principal symbol $p$ of the $h$-pseudodifferential operators associated to $Q$
is given by $p(k,x):=\lambda_1(k)+w_0\big({\frac{x}{|x|}}\big)|x|^{-\delta}$ near $p^{-1}([a,b]).$
By observing that $x\cdot \nabla_x(w_0(\frac{x}{|x|}))=0,$ we get
$$
x\cdot \nabla_x(w_0(\frac{x}{|x|})|x|^{-\delta})=-\delta w_0(\frac{x}{|x|})|x|^{-\delta}.
$$
Then
$$
\{\nabla \lambda_1(k)\cdot x,p(k,x)\}=|\nabla \lambda_1(k)|^2+\delta w_0(\frac{x}{|x|})|x|^{-\delta}\Delta\lambda_1(k)\geq c_0>0
$$
in $p^{-1}([a,b])$, here $\{u,v\}:=\frac{\partial u}{\partial x}\frac{\partial v}{\partial k}-
\frac{\partial u}{\partial k}\frac{\partial v}{\partial x}$ is the poisson bracket between $u$ and $v$ and
$c_0$ is some positive constant.
This shows that every energy $\lambda\in [a,b]$
is non-trapping for the classical Hamiltonian $p(k,x)$ (see for instance \cite[Proposition 21.3]{HiSi96_01}).
Now, by Lemma 3.5 in \cite{RoTa88_01} we get the estimation (\ref{habslimit}).
For more details see \cite{DiZe10_01}.
\end{proof}

\begin{proposition}\label{traceunif}
One has ${\rm tr}\big[{\mathcal Q}(z)\big]={\mathcal
O}\big(h^\infty\big),$ uniformly for $z\in {\mathcal U}_\pm.$
\end{proposition}

\begin{proof}
Recall the expression of ${\mathcal Q}(z)$:
$$
{\mathcal Q}(z)=(z-P_\mu)^{-1}-(z-Q)^{-1}-(z-P_{F_2})^{-1}\widetilde W(z-P_{F_1})^{-1}.
$$
Following the proposition \ref{resolventidentity}, we have
$$
{\mathcal Q}(z)=(z-P_{F_2})^{-1}\widetilde W(z-P_{F_1})^{-1}{\mathcal G}_1(z-Q)^{-1}
+(z-P_\mu)^{-1}{\mathcal G}_2(z-P_{F_2})^{-1}\widetilde W(z-Q)^{-1}.
$$
Therefore ${\rm tr}\big[{\mathcal Q}(z)\big]={\rm tr}\big(I_1(z)\big)+{\rm tr}\big(I_2(z)\big),$ where
\begin{align}
& I_1(z)=(z-P_{F_2})^{-1}\widetilde W(z-P_{F_1})^{-1}{\mathcal G}_1(z-Q)^{-1}\\
\text{and} &{}\nonumber\\
& I_2(z)=(z-P_\mu)^{-1}{\mathcal G}_2(z-P_{F_2})^{-1}\widetilde W(z-Q)^{-1}.\nonumber
\end{align}

Using the cyclicity of the trace, we show that:
\begin{align}
& {\rm tr}\big(I_1(z)\big)={\rm tr}\Big[\langle hx\rangle^{-1}\langle hx\rangle(z-P_{F_2})^{-1}
\widetilde W(z-P_{F_1})^{-1}{\mathcal G}_1(z-Q)^{-1}\Big]\nonumber\\
& ={\rm tr}\Big[\langle hx\rangle(z-P_{F_2})^{-1}\langle hx\rangle^{-1}\cdot\langle hx\rangle
\big(\widetilde W(z-P_{F_1})^{-1}{\mathcal G}_1\big)\langle hx\rangle\cdot\langle hx\rangle^{-1}
(z-Q)^{-1}\langle hx\rangle^{-1}\Big],\nonumber
\end{align}
which implies that
\begin{align}\label{f43}
 \Big|{\rm tr}\big(I_1(z)\big)\Big|\leq& \big\Vert\langle hx\rangle(z-P_{F_2})^{-1}\langle hx\rangle^{-1}\big\Vert\times
 \big\Vert\langle hx\rangle
 \big(\widetilde W(z-P_{F_1})^{-1}{\mathcal G}_1\big)\langle hx\rangle\big\Vert_{\rm tr} \\
&\times\big\Vert\langle hx\rangle^{-1}(z-Q)^{-1}\langle hx\rangle^{-1}\big\Vert.\nonumber
\end{align}

Since $\displaystyle\text{dist}\big(\widetilde W, {\mathcal G}_1\big)\geq \frac{C(M)}{h}.$ Then,
uniformly for $z\in{\mathcal U}_\pm,$
$$
\Big\Vert\langle hx\rangle\big(\widetilde W(z-P_{F_1})^{-1}
{\mathcal G}_1\big)\langle hx\rangle\Big\Vert_{\rm tr}
={\mathcal O}\big(h^\infty\big),
$$
we use here the lemma \ref{negligeabletraceoperator1}. On the other
hand, the proposition \ref{muabslimit2} gives that the third term of
\eqref{f43} is ${\mathcal O}(h^{-1})$ uniformly on $z\in{\mathcal
U}_\pm$ and by pseudodifferential calculus we get 
$$
\Big\Vert\langle
hx\rangle(z-P_{F_2})^{-1}\langle hx\rangle^{-1}\Big\Vert={\mathcal
O}(1).
$$ 
Then we obtain ${\rm tr}\big(I_1(z)\big)={\mathcal O}\big(h^\infty\big)$
uniformly for $z\in{\mathcal U}_\pm.$

For ${\rm tr}\big(I_2(z)\big),$ we use again the resolvent identity:
$$
(z-P_\mu)^{-1}=(z-Q)^{-1}+(z-Q)^{-1}\widetilde W(z-P_\mu)^{-1}.
$$
Therefore ${\rm tr}\big(I_2(z)\big)={\rm tr}\big(I_{21}(z)\big)+{\rm tr}\big(I_{22}(z)\big),$ where
$$
I_{21}(z)=(z-Q)^{-1}{\mathcal G}_2(z-P_{F_2})^{-1}\widetilde W(z-Q)^{-1}
$$
and
$$
I_{22}(z)=(z-Q)^{-1}\widetilde W(z-P_\mu)^{-1}{\mathcal G}_2(z-P_{F_2})^{-1}
\widetilde W(z-Q)^{-1}.
$$
The two terms ${\rm tr}\big(I_{21}(z)\big), {\rm tr}\big(I_{22}(z)\big)$ are in some sense analogous then we discuss only the term
${\rm tr}\big(I_{22}(z)\big).$ For this, we use again the cyclicity of the trace and the fact that
the multiplication operator by
$\langle hx\rangle^{-2}$ is bounded, we obtain:
\begin{align}
{\rm tr}\big(I_{22}(z)\big)
& ={\rm tr}\Big[\langle hx\rangle^{2}\widetilde W \langle hx\rangle
\cdot\langle hx\rangle^{-1}(z-P_\mu)^{-1}\langle hx\rangle^{-1}\nonumber\\
& \cdot\Big(\langle hx\rangle{\mathcal G}_2(z-P_{F_2})^{-1}
\widetilde W\langle hx\rangle^{2}\Big)
\cdot\langle hx\rangle^{-2}(z-Q)^{-2}\langle hx\rangle^{-2}\Big].\nonumber
\end{align}
Since $\displaystyle{\rm dist}\big({\mathcal G}_2,\widetilde W\big)\geq \frac{C(M)}{h}.$
Then, uniformly for $z\in{\mathcal U}_\pm,$
$$
\Big\Vert\langle hx\rangle \big({\mathcal G}_2(z-P_{F_2})^{-1}
\widetilde W\big)\langle hx\rangle^{2}\Big\Vert_{\rm tr}={\mathcal O}\big(h^\infty\big).
$$
See Lemma \ref{negligeabletraceoperator1}. The proposition \ref{muabslimit2} gives the following estimations:
$$
\Big\Vert\langle hx\rangle^{-1}(z-P_\mu)^{-1}\langle hx\rangle^{-1}\Big\Vert={\mathcal O}\big(h^{-1}\big)
,\quad
\Big\Vert\langle hx\rangle^{-2}(z-Q)^{-2}\langle hx\rangle^{-2}\Big\Vert={\mathcal O}\big(h^{-2}\big).
$$
Using the definition of $\widetilde W,$
more precisely, the fact that his support
is in $B\big(0,r_1 M^{-{\frac{1}{\delta}}}\big),$
we get $\Vert\langle hx\rangle^{3}\widetilde W\Vert_\infty={\mathcal O}_M(1).$
Therefore ${\rm tr}\big(I_{22}(z)\big),$ (and in the same way ${\rm tr}\big(I_{21}(z)\big)$), is ${\mathcal
O}\big(h^\infty\big),$ uniformly for
$z\in {\mathcal U}_\pm,$ which finish the proof of Proposition \ref{traceunif}.
\end{proof}

\noindent
{\sl End of the proof of Theorem \ref{dSSFmu}}: Let $\xi_\mu(\lambda)$ (resp. $\xi_h(\lambda)$) be the spectral shift
function associated to the pair $(P_\mu,P_0)$
(resp. $(Q,P_0)$). An immediate consequence of the previous proposition
and Proposition \ref{mainresultofthesubsection}, identity \eqref{f42} is that:
$$
\xi'_\mu(\lambda)=\xi'_h(\lambda)+{\mathcal O}\big(h^\infty\big),
\ \textmd{uniformly for}\  \lambda\in[a,b],
$$
which together with Theorem 2.3 in \cite{DiZe10_01} imply \eqref{strongSSFmu}.

\end{document}